\documentclass[letterpaper,10pt,conference]{ieeeconf}

\IEEEoverridecommandlockouts
\usepackage{amsmath,amsfonts,amssymb}
\usepackage{bm,fixmath}
\usepackage{graphicx,tikz}
\usepackage{pifont}
\usepackage{makecell,booktabs}
\usepackage{cite}
\usepackage{color,xspace}

\DeclareMathOperator*{\argmin}{arg\,min}
\DeclareMathOperator{\prox}{prox}
\newcommand{\qed}{\hfill$\square$}

\newcommand{\norm}[1]{\left\lVert#1\right\rVert}
\newcommand{\expval}[1]{\mathbb{E}\left[#1\right]}

\newcommand{\N}{\mathbb{N}}
\newcommand{\R}{\mathbb{R}}
\newcommand{\kron}{\otimes}

\newcommand{\Em}{\mathbold{E}}

\renewcommand{\Im}{\mathbold{I}}
\newcommand{\Wm}{\mathbold{W}}

\newcommand{\e}{\mathbold{e}}
\newcommand{\x}{\mathbold{x}}
\newcommand{\y}{\mathbold{y}}

\newcommand{\0}{\bm{0}}
\newcommand{\1}{\bm{1}}

\newtheorem{remark}{Remark}
\newtheorem{proposition}{Proposition}

\newtheorem{assumption}{Assumption}
\newtheorem{lemma}{Lemma}
\newtheorem{corollary}{Corollary}

\usepackage{algorithm}
\usepackage{algpseudocode}
\algrenewcommand\algorithmicrequire{\textbf{Input:}}

\newif\ifarxiv
\arxivtrue 

\begin{document}

\title{Distributed and Inexact Proximal Gradient Method \\ for Online Convex Optimization}

\author{Nicola~Bastianello and
        Emiliano~Dall'Anese%
\thanks{N. Bastianello is with the Department of Information Engineering, University of Padova, Italy, and a visiting student at the University of Colorado Boulder;  \texttt{nicola.bastianello.3@phd.unipd.it}
E. Dall'Anese is with the Department of Electrical, Computer and Energy Engineering, University of Colorado, Boulder, CO, USA; \texttt{emiliano.dallanese@colorado.edu}}
\thanks{This work was supported in part by the National Science Foundation CAREER award 1941896.} 
}

\maketitle

\begin{abstract}
This paper develops and analyzes  an online distributed proximal-gradient method (DPGM) for time-varying composite convex optimization problems. Each node of the network features a local cost that includes a smooth strongly convex function and a non-smooth convex  function, both changing over time. By  coordinating  through  a connected communication network, the nodes collaboratively track the trajectory of the minimizers without exchanging their local cost functions. The DPGM is implemented in an online fashion, that is, in a setting where only a limited number of steps are implemented before the function changes. Moreover, the algorithm is analyzed in an inexact scenario, that is, with a source of additive noise, that can represent \emph{e.g.} communication noise or quantization. It is shown that the tracking error of the online inexact DPGM is upper-bounded by a convergent linear system, guaranteeing convergence within a neighborhood of the optimal solution.
\end{abstract}

\section{Introduction and Motivation}\label{sec:motivation}


This paper considers a network of $N$ agents collaboratively solving a \textit{time-varying optimization problem} of the form:
\begin{equation}\label{eq:generic-problem}
\begin{split}
	\x^*(t_k) &= \argmin_{x  \in \R}  \sum_{i = 1}^{N} \left( f_i(x_i;t_k) + g_i(x_i;t_k) \right) \\
	&\text{s.t.} \quad x_i = x_j \quad \forall (i,j) \in \mathcal{E}
\end{split}
\end{equation}
where $f_i$ is a smooth, strongly convex function, $g_i$ is a convex non-smooth functions, and $\{ t_k \}_{k \in \N}$ is a time index. The $x_i \in \R$ are the local states of nodes in the network $\mathcal{G} = (\mathcal{V}, \mathcal{E})$.

Problem~\eqref{eq:generic-problem} is prevalent in learning and data processing problems over networks~\cite{simonetto2017decentralized,dixit2019online,dall2019optimization,zhang2019distributed,SimonettoGlobalsip2014};
temporal variations of the cost capture streams of data/measurements, with a new datum arriving at each interval $T_{\mathrm{s}} := t_{k+1} - t_k$, or time-varying problem parameters. Problem~\eqref{eq:generic-problem} can also model a number of data-driven control tasks, including measurement-based algorithms for network optimization~\cite{Bolognani_feedback_15,Bernstein2019feedback}, predictive control~\cite{paternain2019prediction}, and design of optimal controllers for distributed systems~\cite{dhingra2018proximal}; in this case, measurements are gathered from the physical system at every interval $T_{\mathrm{s}}$ and changes in the control objectives lead to a time-varying problem formulation.  

The goal of this paper is to develop a distributed algorithm that allows the nodes to collaboratively track the optimal trajectory $\{ x^*(t_k) \}$ of the sequence of composite problems~\eqref{eq:generic-problem}. Due to the dynamic nature of problem~\eqref{eq:generic-problem}, the proposed distributed proximal gradient method (DPGM) will be characterized by the application of a limited number of algorithmic steps to each problem, giving rise to an \emph{online} implementation of the algorithm.

The paper further studies DPGM in an \emph{inexact scenario}, characterized by the following nonidealities: \textit{(e1)} approximate  evaluation of the gradient of $f_i$; \textit{(e2)} approximate proximal evaluation; and, \textit{(e3)} state noise. Approximate gradient information naturally captures the case where, for example, bandit or zeroth order methods are utilized to estimate $\nabla_x f_i(x; t_k)$~\cite{Flaxman05,Hajinezhad19}.  An approximate proximal evaluation may emerge when the proximal operator can be performed only up to a given precision~\cite{schmidt2011convergence,salzo2012inexact,barre_principled_2020}. This may be due to nodes with limited  processing power or energy-related concerns. Finally, errors in the states (that is, variables) may be due to communications  or transmissions of quantized vectors, see \textit{e.g.}~\cite{kar_distributed_2009,majzoobi_analysis_2019,reisizadeh_exact_2019,magnusson_maintaining_2019}. The overall setting is stylized in Figure~\ref{fig:inexact-network}.

\begin{figure}[!t]
	\centering
	\includegraphics[width=0.40\textwidth]{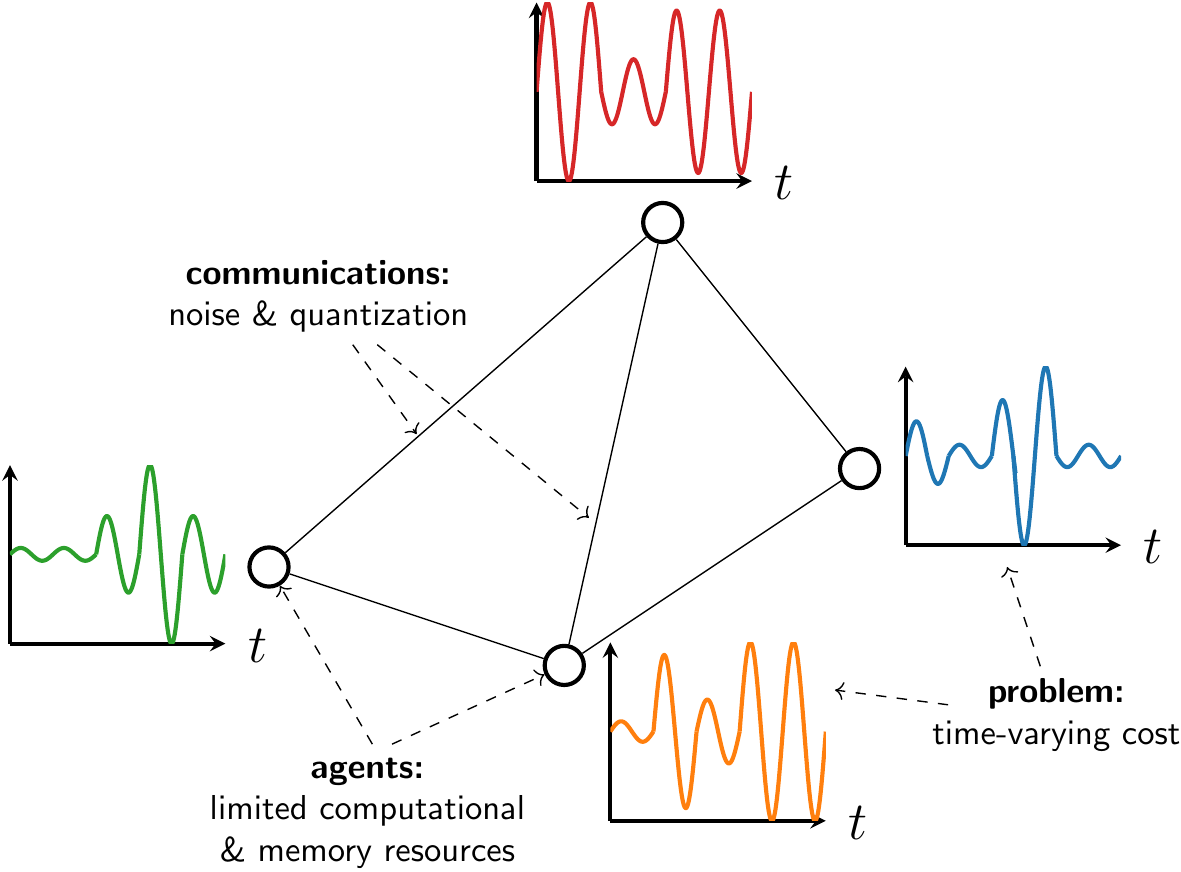}
	\caption{A qualitative illustration  of the distributed, inexact and time-varying framework considered in the paper. Each node is observing time-varying data (\textit{e.g.} measurements with a sensor) which imply that the problem is time-varying. Moreover, communication errors and the limited resources available at each node introduce inexactness in the algorithm's updates.}
	\vspace{-.4cm}
	\label{fig:inexact-network}
\end{figure}

The paper shows that the tracking error is upper-bounded by a convergent linear system; in particular, the iterates generated by the algorithm converge linearly to a neighborhood of the optimal solution trajectory $\{x^*(t_k)\}$.

Prior literature in the context of online distributed algorithms includes, \textit{e.g.}, \cite{zhang2019distributed,hosseini2016online,akbari2015distributed,shahrampour2017distributed,onlineSaddle}. Both \cite{hosseini2016online} and \cite{akbari2015distributed} consider an online sub-gradient framework, and perform a dynamic regret analysis. The recent paper \cite{yuan_can_2020} proposes an online version of distributed gradient descent. (DGD)\cite{yuan_convergence_2016} alongside an online gradient tracking scheme. These works do not address composite costs and involve exact updates. In \cite{shahrampour2017distributed} and  \cite{zhang2019distributed}, two (exact) distributed online algorithms are proposed to solve smooth  optimization problems, under the assumption that there exists a linear model for the optimal trajectory. In this paper, no model for the optimal trajectory is postulated. For smooth cost functions, decentralized online (and exact) prediction-correction schemes were developed in~\cite{simonetto2017decentralized,bastianello_distributed_2020}, an online exact saddle-point algorithm was developed in~\cite{onlineSaddle} for a consensus problem, and a distributed primal-dual algorithm with a star communication topology was developed in~\cite{Bernstein2019feedback}.

Different PGM-based algorithms have been proposed to solve distributed, composite problems in static optimization, for example \cite{chen_fast_2012,aybat_asynchronous_2015,shi_proximal_2015,zeng_fast_2017,li_decentralized_2019}. The approaches of \cite{chen_fast_2012,aybat_asynchronous_2015} require an inertial scheme and a diminishing step-size to guarantee exact convergence, respectively, while the methods proposed in \cite{shi_proximal_2015,zeng_fast_2017,li_decentralized_2019} are based on gradient tracking schemes, which lead to exact convergence with a fixed step-size. The recent papers \cite{alghunaim_decentralized_2019,xu_distributed_2020} propose two different unified frameworks for distributed PGMs based on gradient tracking schemes. Interestingly, linear convergence for this class of algorithms can be guaranteed in the static and exact scenario only provided that the non-smooth part of the cost be common to all nodes.

To the best of the authors' knowledge, the only online distributed proximal gradient method that can handle time-varying costs has been proposed in~\cite{dixit_online_2019}. The DP-OGD algorithm of \cite{dixit_online_2019} can be applied to $B$-connected graphs, but requires that the non-smooth part of the costs be common to all nodes. An interesting feature of DP-OGD in~\cite{dixit_online_2019} is that it alternates consensus steps (\textit{i.e.} rounds of communications) with proximal gradient steps. The algorithm guarantees a sub-linear dynamic regret, provided that the number of communication rounds and the step-size be chosen in a coordinated fashion.

\noindent \emph{Paper organization}.
Section~\ref{sec:time-invariant} introduces the inexact DPGM in a time-invariant scenario. Section~\ref{sec:time-varying} then presents an online version of the inexact DPGM, and studies its convergence. Section~\ref{sec:simulations} provides simulations results.

\noindent \emph{Notation}
$\| \cdot \|$ denotes the Euclidean norm, $x^\top$ denotes transposition, $\kron$ denotes the Kronecker product. Given a symmetric matrix $M$, $\lambda_\mathrm{min}(M)$ denotes the minimum eigenvalue of $M$; if, additionally, $M$ is stable, then $\rho(M) \in (0,1)$ denotes the absolute value of the largest singular value strictly smaller than one.  With $\partial f(x)$ is denoted the subdifferential of a convex function $f$ and by $\tilde{\nabla} f(x) \in \partial f(x)$ a subgradient. The proximal operator $\prox_{\alpha g}: \mathbb{R}^{n} \rightarrow  \mathbb{R}^{n}$ of a convex closed and proper function $g : \R^n \to \R$ is defined as:
$$
    \vspace{-.1cm} \prox_{\alpha g} ( y ) := \arg \min_{x} \left\{ g(x) + \frac{1}{2 \alpha} \|x - y\|^2 \right\},
$$
with $\alpha > 0$. The vectors of all ones and zeros are denoted by $\1$ and $\0$, respectively. In the following, local variables will be denoted by normal case letters, and global variables by boldface letters.

\noindent A sequence $\{ \beta^\ell \}_{\ell \in \N}$ is said to be \textit{R-linearly} convergent if there exist $C > 0$ and $\lambda \in (0,1)$ such that $\beta^\ell \leq C \lambda^\ell$ for any $\ell \in \N$.

\noindent Boldface vectors denote stacked local quantities, \emph{e.g.} $\x = [x_1, \ldots, x_N]^\top$, and $\mathcal{N}_i$ denotes the neighbors of node $i$.

\section{Inexact Composite Optimization}\label{sec:time-invariant}
This section introduces the proposed inexact DPGM for static problems, that is, for which the cost function in~\eqref{eq:generic-problem} does not change during the execution of the algorithm:
\begin{subequations}
\label{eq:time-invariant-problem}
\begin{align}
	 \x^* &= \argmin_{\{x_i \in \R\}_{i = 1}^N} \sum_{i = 1}^{N} \left( f_i(x_i) + g_i(x_i) \right) \label{eq:time-invariant-cost} \\
	&\text{s.t.} \quad x_i = x_j \quad \forall (i,j) \in \mathcal{E} \label{eq:consensus-constraints}
\end{align}
\end{subequations}
where by the consensus constraints~\eqref{eq:consensus-constraints} it holds that $\x^* = \1_N \kron x^*$, with $x^* := \argmin_{x \in \R} \sum_{i = 1}^{N} \left( f_i(x) + g_i(x) \right)$.

The following assumptions are used.

\begin{assumption}\label{as:graph}
The graph $\mathcal{G}$ is undirected and connected, and $\Wm$ is a doubly stochastic consensus matrix for the graph.
\end{assumption}

\begin{assumption}\label{as:local-costs}
The local costs are such that for any $i = 1, \ldots, N$:
\begin{itemize}\setlength\itemsep{0em}
	\item $f_i : \R \to \R$ is $L_{f_i}$-smooth and $m_{f_i}$-strongly convex;
	
	\item $g_i : \R \to \R$ is closed, convex and proper, and $L_{g_i}$-Lipschitz continuous, but possibly non-smooth.
\end{itemize}
In the following we use the notations $L_f = \max_i L_{f_i}$, $m_f = \min_i m_{f_i}$ and $L_g = \max_i L_{g_i}$.
\end{assumption}

\begin{remark}
For simplicity of exposition it is also assumed that the local states are scalar, $x_i \in \R$; the approach straightforwardly  extends to vectors in $\R^n$.
\end{remark}

The DPGM aims to identify the solution of~\eqref{eq:time-invariant-problem} by sequentially performing the following steps:  
\begin{subequations}\label{eq:time-invariant-DPGM}
\begin{align}
	\y^{\ell+1} &= \Wm \x^\ell - \alpha \nabla f(\x^\ell) \label{eq:DPGM-gradient-step} \\
	\x^{\ell+1} &= \prox_{\alpha g}(\y^{\ell+1}) + \e^\ell
\end{align}
\end{subequations}
where $f(\x) = \sum_{i = 1}^N f_i(x_i)$, $g(\x) = \sum_{i = 1}^N g_i(x_i)$, $\ell \in \N$ is the iteration index and  $\alpha > 0$ is the step size. The vector $\e^\ell$ represents random additive noise that satisfies the following assumption.

\begin{assumption}[Error]\label{as:stochastic-error}
The error $\e^\ell$ is the realization of a multi-variate random variable for which there exists $\eta \in [0, +\infty)$ such that $\expval{\norm{\e^\ell}} \leq \eta$.
\end{assumption}

The error can model for example: \textit{(e1)} approximate gradient evaluation; \textit{(e2)} approximate proximal evaluation; and, \textit{(e3)} state noise. Assumption~\ref{as:stochastic-error} is verified by random vectors that have finite mean and covariance, as shown in the following lemma.

\begin{lemma}[Expectation of error norm]\label{lem:expectation_norm}
Let $\e$ be a random vector with finite mean $\bm{\mu}$ and finite covariance matrix $\bm{\Sigma}$. Then, one has that
\begin{align}
	\expval{\norm{\e}} \leq \eta := \sqrt{\operatorname{tr}(\bm{\Sigma}) + \norm{\bm{\mu}}^2} < +\infty.
\end{align}
\end{lemma}

\noindent \emph{Proof}. \ifarxiv See Appendix~\ref{app:lemma-proofs}. \else See Appendix~I in \cite{bastianello_distributed_2020_2}. \fi

The additive error $\e^\ell$ is not the only source of inexactness for~\eqref{eq:time-invariant-DPGM}. Indeed, similarly to DGD~\cite{yuan_convergence_2016}, algorithm~\eqref{eq:time-invariant-DPGM} can identify a solution of~\eqref{eq:time-invariant-problem} only up to a precision error. This is due to the fact that DPGM is actually solving the following relaxed version of~\eqref{eq:time-invariant-problem}
\begin{equation}\label{eq:regularized-problem}
	\tilde{\x} := \argmin_{\x \in \R^{N}} \left\{ \varphi_\alpha(\x) + \alpha g(\x) \right\}
\end{equation}
where $\varphi_\alpha(\x) := (1/2) \x^\top (\Im - \Wm) \x + \alpha f(\x)$ which relaxes the consensus constraints~\eqref{eq:consensus-constraints} using the quadratic function $(1/2) \x^\top (\Im - \Wm) \x$. Notice that in general, $\tilde{\x}$ does not belong to the consensus subspace $\operatorname{span}\{ \1 \}$. It is now possibly to see that DPGM is a proximal gradient method with unitary step-size applied to~\eqref{eq:regularized-problem}.

\begin{lemma}[Relaxed problem]\label{lem:properties-gamma}
The function $\varphi_\alpha : \R^{N} \to \R$ is $L_\varphi$-smooth and $m_\varphi$-strongly convex, with
$$
	L_\varphi = 1 - \lambda_\mathrm{min}(\Wm) + \alpha L_f \quad \text{and} \quad m_\varphi = \alpha m_f.
$$
\end{lemma}
\emph{Proof.} \ifarxiv See Appendix~\ref{app:lemma-proofs}. \else See Appendix~I in \cite{bastianello_distributed_2020_2}. \fi

In the absence of additive noise, the convergence of the algorithm to $\tilde{\x}$ is guaranteed imposing the following condition on the step-size:
\begin{equation}
\label{eq:alpha-relaxed-static}
	\alpha \in \left( 0, \frac{1 + \lambda_\mathrm{min}(\Wm)}{L_f} \right)
\end{equation}
which readily follows from the condition $1 \leq 2 / L_\varphi$. If $\alpha$ satisfies~\eqref{eq:alpha-relaxed-static}, then the algorithm converges Q-linearly to the solution of the regularized problem~\eqref{eq:regularized-problem}; that is, \cite{taylor_exact_2018}
\begin{equation}\label{eq:linear-convergence-pgm}
	\norm{\x^{\ell+1} - \tilde{\x}} \leq \zeta_\varphi \norm{\x^\ell - \tilde{\x}}
\end{equation}
where $\zeta_\varphi := \max\left\{ |1 - L_\varphi|, |1 - m_\varphi| \right\} \in (0,1)$. Obviously, $\norm{\x^{\ell} - \tilde{\x}} \rightarrow 0$ as $\ell \rightarrow + \infty$.

\subsection{Convergence analysis}
\label{sec:convergence_static}

The result of this section establishes the convergence of the inexact DPGM  to a neighborhood of the optimal solution $\x^*$ of problem~\eqref{eq:time-invariant-problem}. The size of the neighborhood will be shown to depend both on the inexactness introduced by update~\eqref{eq:time-invariant-DPGM} and the structure of the approximate problem~\eqref{eq:regularized-problem}.

In the following, the average of the local variables $\{x_i^\ell\}_{i = 1}^N$ at iteration $\ell$ is denoted by $\bar{x}^\ell := (1/N) \sum_{i = 1}^N x_i^\ell$, and let $\bar{\x}^\ell :=  \1_N \kron \bar{x}^\ell$. For future developments, notice that $\bar{\x}^\ell$ can also be written as $\bar{\x}^\ell = \frac{1}{N} \1_N \1_N^\top \x^\ell$. The following convergence result is related to the evolution of the error vector 
\begin{equation}
d^\ell := \Big[ \norm{\bar{\x}^\ell - \x^*}, \norm{\x^\ell - \bar{\x}^\ell}, \norm{\x^\ell - \tilde{\x}} \Big]^\top.
\end{equation}

\begin{proposition}[Time-invariant convergence]\label{pr:time-invariant-convergence}
Let Assumptions~\ref{as:graph}--\ref{as:stochastic-error} hold, and let the step size $\alpha$ verify
\begin{equation}\label{eq:step-size-condition}
	0 < \alpha < \min\left\{ \frac{1 + \lambda_\mathrm{min}(\Wm)}{L_f}, \frac{2}{L_f + m_f} \right\};
\end{equation}
then algorithm~\eqref{eq:time-invariant-DPGM} converges R-linearly to a neighborhood of the optimal solution $\x^*$. 

\noindent In particular, the dynamics of the mean error $\expval{\norm{\x^{\ell+1} - \x^*}}$ can be bounded as follows:
\begin{align}
	&\expval{d^{\ell+1}} \leq A \expval{d^\ell} + b + \1_3 \expval{\norm{\e^\ell}} \label{eq:linear-error-system} \\
	&\expval{\norm{\x^{\ell+1} - \x^*}} \leq \begin{bmatrix} 1 & 1 & 0 \end{bmatrix} \expval{d^{\ell+1}},
\end{align}
where
$$
    A := \begin{bmatrix} c & \alpha L_f & 0 \\ 0 & \rho(\Wm) & \alpha L_f \\ 0 & 0 & \zeta_\varphi \end{bmatrix}, \quad
    b := \begin{bmatrix} 2 \alpha L_g \\ 2 \alpha L_g + \norm{(\Im - \Wm)\tilde{\x}} \\ 0 \end{bmatrix}
$$
with $c := \sqrt{1 - 2\alpha m_f L_f/(m_f + L_f)} \in (0,1)$, and the inequality holds entry-wise.
\end{proposition}

Convergence to a neighborhood of $\x^*$ follows by noticing that all the eigenvalues of the matrix $A$ are strictly inside the unitary circle; that is, the sequence of errors is upper-bounded by a convergent linear system with a fixed input.

The proof of Proposition~\ref{pr:time-invariant-convergence} is reported in \ifarxiv Appendix~\ref{app:ti-convergence}, \else Appendix~II of \cite{bastianello_distributed_2020_2}, \fi alongside some auxiliary Lemmas.

\section{Online Composite Optimization}\label{sec:time-varying}
The paper now turns to the \emph{time-varying} problem~\eqref{eq:generic-problem}, under the following assumption.

\begin{assumption}\label{as:time-varying}
The sequence of problems~\eqref{eq:generic-problem} is defined over a fixed graph $\mathcal{G}$ that satisfies Assumption~\ref{as:graph}. Moreover, at each time $\{ t_k \}_{k \in \N}$ the local costs $f_i(x;t_k)$ and $g_i(x;t_k)$ satisfy Assumption~\ref{as:local-costs}.
\end{assumption}

Assume that, because of underlying communication and computation bottlenecks, a limited number of iterations and communication rounds can be performed over an interval $T_{\mathrm{s}}$; hence, each problem  can be solved only approximately (representing an additional source of inexactness for the proposed algorithm). Denote as $M_{\mathrm{o}} > 0$ the number of algorithmic steps. The \emph{online inexact DPGM} is then described by Algorithm~\ref{alg:time-varying-DPGM}.

\begin{algorithm}[!ht]
\caption{Online inexact DPGM}
\label{alg:time-varying-DPGM}
\begin{algorithmic}[1]
	\Require $x_i(t_0)$, $i = 1,2,\ldots,N$,  $\alpha$, consensus matrix $\Wm$.
	\For{$k = 1, 2, \ldots$, each node}
		\Statex\hspace*{\algorithmicindent}{\color{blue}// Observe new problem}
		\State Observe $f_i(\cdot;t_k)$ and $g_i(\cdot;t_k)$
		\Statex\hspace*{\algorithmicindent}{\color{blue}// Apply solver}
		\State Set $x_i^0 = x_i(t_{k-1})$
		\For{$\ell=0,1,\ldots, M_{\mathrm{o}}-1$ each agent $i$}
			\Statex\hspace*{3em}{\color{blue}// Communication}
			\State Transmit $x_i^\ell$ to  neighbors $\mathcal{N}_i$ 
			\State Receive $\hat{x}_j^\ell$ from  neighbors $\mathcal{N}_i$ \Comment{comm. noise}
			\Statex\hspace*{3em}{\color{blue}// Proximal gradient step}
			\State Compute $\hat{\nabla} f_i(x_i^\ell; t_k)$ \Comment{inexact gradient}
			\State Compute the steps: \Comment{inexact proximal}
			\begin{align*}
				y_i^{\ell+1} &= \sum_{j \in \mathcal{N}_i } w_{ij} \hat{x}_j^\ell + w_{ii} x_i^\ell - \alpha \hat{\nabla} f_i(x_i^\ell; t_k)\\
				x_i^{\ell+1} & \approx \prox_{\alpha g_i(\cdot; t_k)}(y_i^{\ell+1})
			\end{align*}
		\EndFor
		\State set $x_i(t_k) = x_i^{M_{\mathrm{o}}}$
	\EndFor
\end{algorithmic}
\end{algorithm}

Notice that in Algorithm~\ref{alg:time-varying-DPGM} a second set of local states, $x_i(t_k)$, $i = 1, \ldots, N$, was introduced. These represent the approximate solution computed by each node after applying $M_{\mathrm{o}}$ steps of the inexact DPGM to the cost observed at time $t_k$. Further, Algorithm~\ref{alg:time-varying-DPGM} is general enough to cover the cases where the functional form of $f_i(\cdot;t_k)$ can be observed, or when only $\hat{\nabla} f_i(x_i^\ell)$ is available.  Moreover, the approximate solution to the problem at time $t_{k-1}$ is used to warm-start\footnote{That is, the initial condition for DPGM at time $t_k$ is chosen equal to the output of DPGM applied to the previous problem at time $t_{k-1}$.} the DPGM applied to problem at time $t_k$.

The sequence $\{\x^*(t_k)\}_{k \in \N}$ represents the unique optimal \emph{trajectory} of~\eqref{eq:generic-problem}. The key question posed here pertains to the ability of the online inexact DPGM to \emph{track} $\{\x^*(t_k)\}_{k \in \N}$, which will be the focus of the following two sections. Notice that in Algorithm~\ref{alg:time-varying-DPGM} the different sources of inexactness (on communication, gradient and proximal) are spelled out; depending on the application, only some of them may be present. Remark~\ref{rem:sources-inexactness} will discuss the effect of these sources on the convergence.

The following section analyzes the convergence of Algorithm~\ref{alg:time-varying-DPGM}, and see \ifarxiv Appendix~\ref{app:time-varying} \else Appendix~III of \cite{bastianello_distributed_2020_2} \fi for the proofs.

\subsection{Convergence analysis}

The temporal  variability  of  the  problem~\eqref{eq:generic-problem}  could  be  measured  based  on ``how fast''  $\x^*(t_k)$ varies~\cite{dall2019optimization,dixit2019online}; more precisely, since $\x^* (t_k)$ is finite and unique (by Assumption~\ref{as:local-costs}), a pertinent measure can be $\norm{\x^*(t_{k+1}) - \x^*(t_k)}$. Accordingly, the following assumption is introduced.

\begin{assumption}\label{as:distance-optima}
Assume that there exists a non-negative scalar $\sigma < +\infty$ such that, at any time $t_k$, $k \in \N$, one has that
\begin{subequations}\label{eq:sigma}
\begin{align}
	\norm{\x^*(t_{k+1}) - \x^*(t_k)}, \ \norm{\tilde{\x}(t_{k+1}) - \tilde{\x}(t_k)} \leq \sigma
\end{align}
\end{subequations}
where $\tilde{\x}(t_k) := \argmin_{\x \in \R^{n N}} \left\{ \varphi_\alpha(\x;t_k) + \alpha g(\x;t_k) \right\}$ and $\varphi_\alpha(\x;t_k) := (1/2) \x^\top (\Im - \Wm) \x + \alpha f(\x;t_k)$.
\end{assumption}

Although each  problem observed at time $t_k$ is solved only approximately (because of the limited number of steps -- $M_\mathrm{o}$ -- applied within an interval $[t_k,t_{k+1})$) and using inexact steps for the DPGM, the following proposition will show that the sequence of the errors  
\begin{align*}
d(t_k) := [ & \norm{\bar{\x}(t_k) - \x^*(t_k)}, \\ & \quad \norm{\x(t_k) - \bar{\x}(t_k)}, \norm{\x(t_k) - \tilde{\x}(t_k)} ]^\top
\end{align*}
does not grow unbounded, but that Algorithm~\ref{alg:time-varying-DPGM} converges to a neighborhood of the optimal trajectory.

\begin{proposition}[Time-varying convergence]\label{pr:time-varying-convergence}
Let Assumptions~\ref{as:time-varying},~\ref{as:distance-optima} hold, and suppose that the step size $\alpha$ satisfies~\eqref{eq:step-size-condition}. Then, Algorithm~\ref{alg:time-varying-DPGM} converges R-linearly to a neighborhood of the optimal solution.

In particular, the mean distance from the optimal trajectory -- $\expval{\norm{\x(t_k) - \x^*(t_k)}}$ -- can be bounded using the following convergent linear system:
\begin{equation}\label{eq:time-varying-error-system}
\begin{split}
	&\expval{d(t_{k+1})} \leq A^{M_{\mathrm{o}}} \expval{d(t_k)} + b'' \\
	&\expval{\norm{\x(t_{k+1}) - \x^*(t_{k+1})}} \leq \begin{bmatrix} 1 & 1 & 0 \end{bmatrix} \expval{d(t_{k+1})}
\end{split}
\end{equation}
where
$$
	b'' := \sum_{\ell = 0}^{M_{\mathrm{o}}-1} A^{M_{\mathrm{o}} - \ell - 1} \Big( \begin{bmatrix} 2\alpha L_g \\ 2\alpha L_g + \sigma' \\ 0 \end{bmatrix} + \eta \1_3 \Big) +  A^{M_{\mathrm{o}}} \begin{bmatrix} \sigma \\ 0 \\ \sigma \end{bmatrix}
$$
and $\sigma' := \sup_{k \in \N} \norm{(\Im - \Wm) \tilde{\x}(t_k)}$.
\end{proposition}

The following Corollary provides an upper bound to the asymptotic error of Algorithm~\ref{alg:time-varying-DPGM}.

\begin{corollary}[Asymptotic error bound]\label{cor:asymptotical-error}
Let Assumptions~\ref{as:time-varying},~\ref{as:distance-optima} hold, suppose that the step size $\alpha$ satisfies~\eqref{eq:step-size-condition}, and let 
$$
    \delta := \max\left\{ c, \rho(\Wm), \zeta_\varphi \right\} \in (0,1).
$$

The asymptotic error of Algorithm~\ref{alg:time-varying-DPGM} under Assumption~\ref{as:stochastic-error} can be bounded as:
\begin{equation}\label{eq:bound_time_var_det}
\begin{split}
	&\limsup_{k \to \infty} \expval{\norm{\x(t_k) - \x^*(t_k)}} \leq \\ &\frac{1}{1 - \delta^{M_{\mathrm{o}}}} \Bigg[ \sigma \delta^{M_{\mathrm{o}}} + \frac{1 - \delta^{M_{\mathrm{o}}+1}}{1 - \delta} \Big( 4 \alpha L_g + \sigma' + 2 \eta \Big) \Bigg].
\end{split}
\end{equation}
\end{corollary}

The form of the bound~\eqref{eq:bound_time_var_det} is similar to the ones in existing works for centralized, exact, and online methods; see, \textit{e.g.},~\cite{dall2019optimization,dixit2019online,Bernstein2019feedback}. However, in the setting of this paper, the bound~\eqref{eq:bound_time_var_det} includes the additional term $\norm{(\Im - \Wm) \tilde{\x}(t_k)}$ that is due to the relaxation~\eqref{eq:regularized-problem}; indeed, since $\tilde{\x} \notin \operatorname{span}\{ \1 \}$, the term  $\norm{(\Im - \Wm) \tilde{\x}(t_k)}$ is always positive. The bound also shows the effect of the errors in the algorithm.

\begin{remark}[Sources of inexactness]\label{rem:sources-inexactness}
Propositions~\ref{pr:time-invariant-convergence}~and~\ref{pr:time-varying-convergence} prove convergence with a generic additive noise as introduced in~\eqref{eq:time-invariant-DPGM}. The different sources of inexactness can also be spelled out as in the following
\begin{subequations}\label{eq:inexact-DPGM-sources}
\begin{align}
	\y^{\ell+1} &= \Wm \left( \x^\ell + \e_{\mathrm{s}}^\ell \right) - \alpha \left( \nabla f(\x^\ell;t_k) + \e_{\mathrm{g}}^\ell \right) \\
	\x^{\ell+1} &= \prox_{\alpha g(\cdot;t_k)}(\y^{\ell+1}) + \e_{\mathrm{p}}^\ell, \label{eq:inexact-proximal-update}
\end{align}
\end{subequations}
where $\e_{\mathrm{g}}^\ell$, $\e_{\mathrm{p}}^\ell$ and $\e_{\mathrm{s}}^\ell$ are the errors affecting the gradient, proximal, and states at iteration $\ell$, respectively. With this explicit representation of the different inexactness sources, it is not difficult to derive from Corollary~\ref{cor:asymptotical-error}
\begin{align*}
	&\limsup_{k \to \infty} \expval{\norm{\x(t_k) - \x^*(t_k)}} \leq \frac{1}{1 - \delta^{M_{\mathrm{o}}}} \Bigg[ \sigma \delta^{M_{\mathrm{o}}} + \\ &+ \frac{1 - \delta^{M_{\mathrm{o}}+1}}{1 - \delta} \Big( 4 \alpha L_g + \sigma' + 2 \left( \eta_{\mathrm{s}} + \alpha \eta_{\mathrm{g}} + \eta_{\mathrm{p}} \right) \Big) \Bigg],
\end{align*}
where $\eta_{\mathrm{s}}, \eta_{\mathrm{g}}, \eta_{\mathrm{p}}$ are the bounds to the norms of the state, gradient, and proximal mean errors, respectively.
\end{remark}

\section{Numerical Results}\label{sec:simulations}

\subsection{Simulation setup}\label{subsec:simulations-setup}
The simulations are performed for a random graph with $N = 25$ nodes and $\sim 160$ edges. The consensus matrix $W$ is built using the Metropolis-Hastings rule. The nodes are tasked with solving, in a distributed fashion, a sparse linear regression problem; that is, 
$f_i$ and $g_i$ are:  
$$
	f_i(x_i;t_k) = \frac{1}{2} \norm{A_{i,k} x_i - b_{i,k}}^2 \ \ \text{and} \ \ g_i(x_i;t_k) = \lambda_1 \norm{x_i}_1.
$$
Let $b_{i,k} = A_{i,k} y(t_k) + e_{i,k}$ be the noisy measurements of the sparse signal $y(t_k)$ performed by the $i$-th node, with $e_{i,k} \in \mathcal{N}(0,10^{-3})$. The signal has $\lfloor n / 2 \rfloor$ non-zero components, and $\lambda_1$ is set to  $\lambda_1 = 0.01$. Different regression matrices $A_{i,k}$ are randomly generated at each sampling time $t_k$, with condition number of $\sim 100$. The signal $y(t_k)$ has sinusoidal components with different phases uniformly drawn from $[0,\pi]$, angular frequency $0.5$, and the sampling time is $T_\mathrm{s} = 0.01$.

The results presented are averaged over $100$ Monte Carlo iterations. As a performance metric, the \textit{cumulative tracking error} is utilized, which is defined as:
$$
	E_k := \frac{1}{k} \sum_{h = 0}^k \norm{\x(t_h) - \x^*(t_h)}.
$$

The nodes exhibit \textit{errors} caused by additive Gaussian noise on the states, with variance $\sigma_{\mathrm{s}}^2$. The step-sizes of the algorithms implemented in the following numerical results are hand-tuned to achieve the best results for each of them.

\subsection{Results}
A first result is presented in Figure~\ref{fig:TV-comparison}, which illustrates the cumulative tracking error attained by DPGM, PG-EXTRA \cite{shi_proximal_2015}, and NIDS \cite{li_decentralized_2019} for different values of $M_{\mathrm{o}}$; that is, by varying the number of steps of the algorithm within each interval $T_{\mathrm{s}}$.

\begin{figure}[!ht]
\centering
	\includegraphics[scale=0.54]{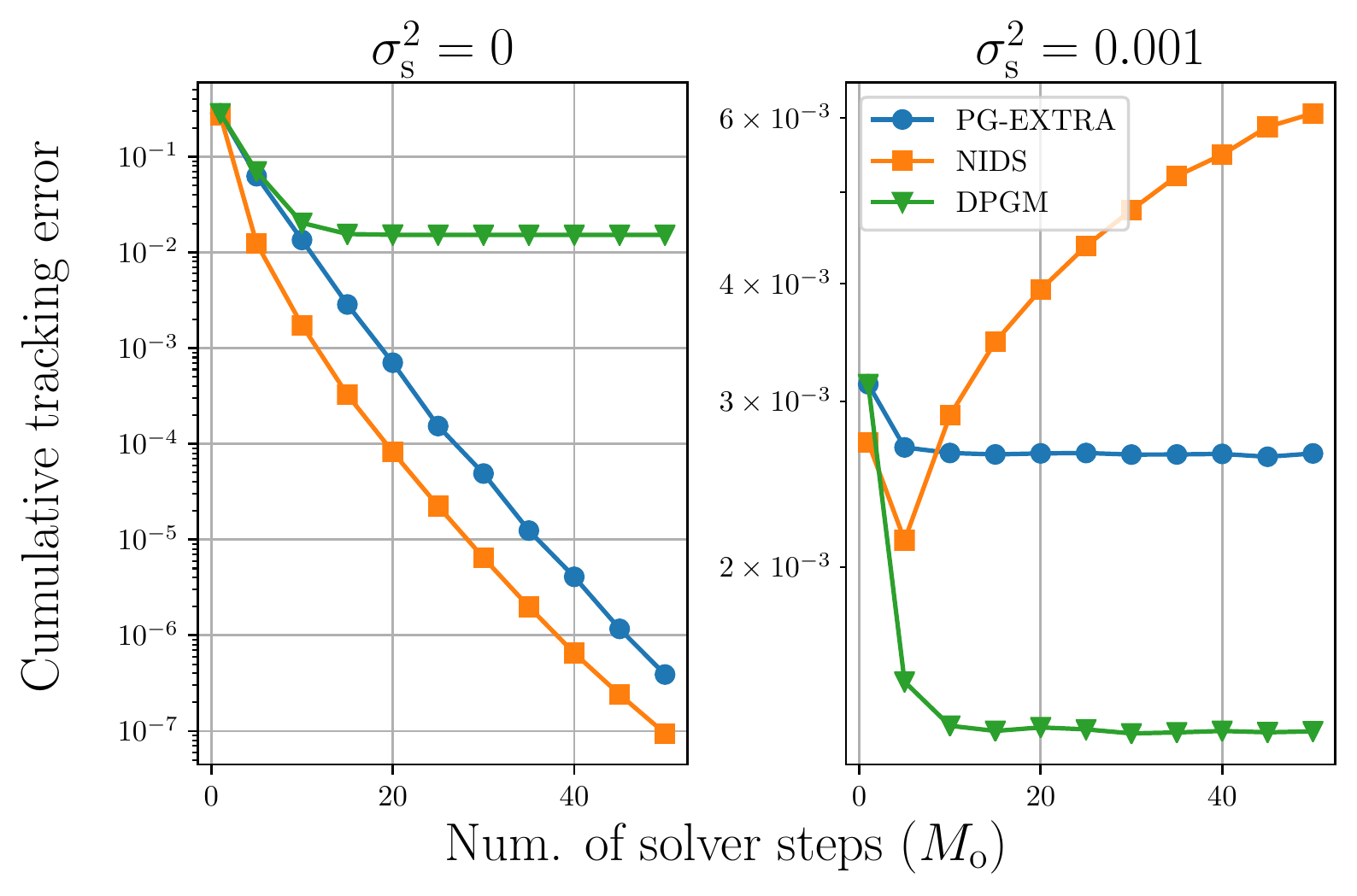}
\caption{Comparison in terms of cumulative tracking error of DPGM (proposed in this paper), PG-EXTRA \cite{shi_proximal_2015}, and NIDS \cite{li_decentralized_2019} for a time-varying sparse linear regression problem, without and with state errors.}
\label{fig:TV-comparison}
\end{figure}

It can be noticed that, in the case of exact algorithmic steps (left plot), PG-EXTRA and NIDS have better performances the larger $M_{\mathrm{o}}$ is, since they converge exactly. On the other hand, when inexactness is introduced and $M_{\mathrm{o}} > 1$, PG-EXTRA attains worse errors than DPGM, while NIDS diverges. Interestingly, when $M_{\mathrm{o}} = 1$, NIDS outperforms DPGM, while PG-EXTRA has the same performance since in this case it reduces to DPGM\footnote{Notice that in \cite{yuan_can_2020} a similar behavior has been observed for smooth, online optimization (with exact algorithmic updates): under the right conditions, inexactly converging primal methods (like DPGM) can outperform exactly converging gradient tracking schemes (as PG-EXTRA and NIDS).}.

Another interesting observation is that  the cumulative error of the proposed DPGM -- as well as PG-EXTRA in the inexact case (right plot) -- decreases only up to a threshold value of $M_{\mathrm{o}}$, and afterwards exhibits a plateau. The following observation explains this behavior. By Corollary~\ref{cor:asymptotical-error}, one  has that
\begin{align}
	&\limsup_{k \to \infty} \expval{\norm{\x(t_k) - \x^*(t_k)}} \leq \frac{\sigma \delta^{M_{\mathrm{o}}}}{1 - \delta^{M_{\mathrm{o}}}} + \nonumber \\ &+ \frac{1}{1 - \delta} \frac{1 - \delta^{M_{\mathrm{o}}+1}}{1 - \delta^{M_{\mathrm{o}}}} \Big( 4 \alpha L_g + \sigma' + 2 \expval{\norm{\e_{\mathrm{s}}^\ell}} \Big) \nonumber \\
	&\simeq \frac{\sigma \delta^{M_{\mathrm{o}}}}{1 - \delta^{M_{\mathrm{o}}}} + \frac{1}{1 - \delta} \Big( 4 \alpha L_g + \sigma' + 2 \expval{\norm{\e_{\mathrm{s}}^\ell}} \Big). \label{eq:time-varying-asymptotical-error}
\end{align}
The right-hand-side of~\eqref{eq:time-varying-asymptotical-error} is therefore the sum of two terms, $\frac{\sigma \delta^{M_{\mathrm{o}}}}{1 - \delta^{M_{\mathrm{o}}}}$, which decreases as $M_{\mathrm{o}}$ increases, and $\frac{1}{1 - \delta} \Big( 4 \alpha L_g + \sigma' + 2 \expval{\norm{\e_{\mathrm{s}}^\ell}} \Big)$, which is constant even if the number of steps $M_{\mathrm{o}}$ varies. Therefore, when the second term becomes dominant over the first one, the cumulative error plateaus.

\begin{table}[!ht]
    \centering
    \caption{Cumulative tracking error for different graph topologies.}
    \label{tab:topologies}
    \begin{tabular}{ccc}
        \textbf{topology} & \textbf{DPGM} & \textbf{PG-EXTRA} \\
        \hline
        \emph{star} & $3.602 \times 10^{-3}$ & $2.799 \times 10^{-3}$ \\
        \emph{circle} & $1.555 \times 10^{-3}$ & $1.756 \times 10^{-3}$ \\
        \emph{circulant (5)} & $7.281 \times 10^{-4}$ & $1.335 \times 10^{-3}$ \\
        \emph{circulant (10)} & $5.736 \times 10^{-4}$ & $1.164 \times 10^{-3}$ \\
        \emph{complete} & $5.526 \times 10^{-4}$ & $1.107 \times 10^{-3}$
    \end{tabular}
\end{table}

Finally, the effect of different network topologies on the performance of DPGM and PG-EXTRA was evaluated, and the results are reported in Table~\ref{tab:topologies}. As one can observe, in the presence of state error, more connected graphs yield smaller cumulative tracking errors. Moreover, DPGM outperforms PG-EXTRA except for the case of a star topology.

\ifarxiv

\appendices
\section{Proofs of Lemmas in Section~\ref{sec:time-invariant}}\label{app:lemma-proofs}

\subsection{Proof of Lemma~\ref{lem:expectation_norm}}

By definition of covariance matrix it holds that
\begin{align}
	\operatorname{tr}(\bm{\Sigma}) &= \expval{\norm{\e}^2} - 2 \expval{\langle \e, \bm{\mu} \rangle} + \norm{\bm{\mu}}^2 \nonumber \\
	& = \expval{\norm{\e}^2} - \norm{\bm{\mu}}^2 \label{eq:trace-covariance}
\end{align}
where the fact $\expval{\langle \e, \bm{\mu} \rangle} = \norm{\bm{\mu}}^2$ -- consequence of the linearity of the expected value -- was used for the second equality. Rearranging~\eqref{eq:trace-covariance}  yields
\begin{equation}\label{eq:bound-square-norm}
	\expval{\norm{\e}^2} = \operatorname{tr}(\bm{\Sigma}) + \norm{\bm{\mu}}^2 < +\infty.
\end{equation}
Moreover, since $\sqrt{\cdot}$ is a concave function, the Jensen inequality holds and one has:
\begin{equation}\label{eq:jensen-inequality}
	\expval{\norm{\e}} = \expval{\sqrt{\norm{\e}^2}} \leq \sqrt{\expval{\norm{\e}^2}}.
\end{equation}
Combining~\eqref{eq:bound-square-norm} and~\eqref{eq:jensen-inequality} proves the Lemma. \qed

\subsection{Proof of Lemma~\ref{lem:properties-gamma}}
The result can be obtained starting as follows
\begin{align*}
	\norm{\nabla \varphi_\alpha(\x) - \nabla \varphi_\alpha(\y)} &\leq \norm{\Im - \Wm} \norm{\x - \y} + \alpha L_f \norm{\x - \y} \\
	&\leq (1 - \lambda_\mathrm{min}(\Wm) + \alpha L_f) \norm{\x - \y}
\end{align*}
where the triangle inequality and the smoothness of $f$ were used.

The strong convexity follows from the strong convexity of $f$ and the fact that $\Im - \Wm$ is positive semidefinite. \qed

\section{Proof of Proposition~\ref{pr:time-invariant-convergence}}\label{app:ti-convergence}

The proof of Proposition~\ref{pr:time-invariant-convergence} relies on the following Lemmas, which are stated and proved for the exact DPGM, setting $\e^\ell = \0$ in~\eqref{eq:time-invariant-DPGM}. Building on these auxiliary results, one can then prove the convergence of the inexact DPGM.

\subsection{Auxiliary results}

\begin{lemma}[Implicit update]\label{lem:implicit-update}
Algorithm~\eqref{eq:time-invariant-DPGM} can be characterized by the following implicit update
\begin{equation}\label{eq:implicit-update}
	\x^{\ell+1} = \Wm \x^\ell - \alpha \left( \nabla f(\x^\ell) + \tilde{\nabla} g(\x^{\ell+1}) \right)
\end{equation}
where $\tilde{\nabla} g(\x^{\ell+1}) \in \partial g(\x^{\ell+1})$ is a subgradient of $g$.
\end{lemma}
\begin{proof}
By the definition of proximal operator, it holds that $\x^{\ell+1} = \prox_{\alpha g}(\y^{\ell+1})$ if and only if $\y^{\ell+1} - \x^{\ell+1} \in \alpha \partial g(\x^{\ell+1})$, which implies that there exists a subgradient $\tilde{\nabla} g(\x^{\ell+1}) \in \partial g(\x^{\ell+1})$ such that $\y^{\ell+1} = \x^{\ell+1} + \alpha \tilde{\nabla} g(\x^{\ell+1})$, and~\eqref{eq:implicit-update} follows.
\end{proof}

\begin{lemma}[Bounded subgradients]\label{lem:bounded-subgradient}
The norm of the subgradient $\nabla f(\x^\ell) + \tilde{\nabla} g(\x^{\ell+1})$ in~\eqref{eq:implicit-update} can be bounded as:
\begin{align}
	\norm{\nabla f(\x^\ell) + \tilde{\nabla} g(\x^{\ell+1})} &\leq L_f \norm{\x^\ell - \tilde{\x}} + \\ &\qquad + 2 L_g + \frac{1}{\alpha} \norm{(\Im - \Wm) \tilde{\x}} \nonumber.
\end{align}
\end{lemma}
\begin{proof}
By the optimality condition for the regularized problem~\eqref{eq:regularized-problem} it holds that $\nabla f(\tilde{\x}) + \tilde{\nabla} g(\tilde{\x}) + (1/\alpha) (\Im - \Wm) \tilde{\x} = 0$ for any subgradient $\tilde{\nabla} g(\tilde{\x}) \in \partial g(\tilde{\x})$. Therefore the following chain of inequalities holds:
\begin{align}
	&\norm{\nabla f(\x^\ell) + \tilde{\nabla} g(\x^{\ell+1})} \nonumber \\
	&= \big\| \nabla f(\x^\ell) + \tilde{\nabla} g(\x^{\ell+1}) - \nabla f(\tilde{\x}) - \tilde{\nabla} g(\tilde{\x}) + \\ &\hspace{3cm} - (1/\alpha) (\Im - \Wm)\tilde{\x} \big\| \nonumber \\
	&\leq \norm{\nabla f(\x^\ell) - \nabla f(\tilde{\x})} + \norm{\tilde{\nabla} g(\x^{\ell+1})} + \nonumber \\ &\hspace{3cm} + \norm{\tilde{\nabla} g(\tilde{\x})} + (1/\alpha) \norm{(\Im - \Wm)\tilde{\x}} \nonumber \\
	&\leq L_f \norm{\x^\ell - \tilde{\x}} + 2 L_g + (1/\alpha) \norm{(\Im - \Wm)\tilde{\x}} \label{eq:intermediate-gradient-bound}
\end{align}
where the triangle inequality was applied for the first inequality, and Lipschitz continuity of the gradient of $f$ and of $g$ for the second inequality.
\end{proof}

\begin{lemma}[Bounded distance from average]\label{lem:distance-from-average}
The distance between the states and the average can be upper bounded as:
\begin{align}
	\norm{\x^{\ell+1} - \bar{\x}^{\ell+1}} &\leq \rho(\Wm) \norm{\x^\ell - \bar{\x}^\ell} + \label{eq:distance-from-average} \\ &\qquad + \norm{\alpha \left( \nabla f(\x^\ell) + \tilde{\nabla} g(\x^{\ell+1}) \right)} \nonumber
\end{align}
where $\rho(\Wm) \in (0,1)$ is the absolute value of the largest singular value strictly smaller than one.
\end{lemma}
\begin{proof}
For simplicity, denote $\Em := \bm{1}_N\bm{1}_N^\top / N$. Using~\eqref{eq:implicit-update}, one can write the update for the distance from the average as:
\begin{align}
	\x^{\ell+1} - \bar{\x}^{\ell+1} &= \Wm \x^\ell - \Em \Wm \x^\ell + \label{eq:evolution-distance-average} \\ &- \alpha (\Im_{N} - \Em) \left( \nabla f(\x^\ell) + \tilde{\nabla} g(\x^{\ell+1}) \right). \nonumber
\end{align}

\noindent One can observe the following two facts:
\begin{enumerate}
	\item matrix $\Im_{N} - \Em =: \bm{\Pi}_{\{ \1 \}^\perp}$ is the projection onto the space orthogonal to the consensus space $\operatorname{span} \{ \1 \}$, and thus it verifies $(\x^\ell - \bar{\x}^\ell) \perp \1_{N}$, $\forall \ell \in \mathbb{N}$;
	\item given that $\Wm$ and $\Em$ commute (due to the double stochasticity of $\Wm$), then it holds
	$$
		\Wm \x^\ell - \Em \Wm \x^\ell = \Wm (\Im_{N} - \Em) \x^\ell = \Wm (\x^\ell - \bar{\x}^\ell).
	$$
\end{enumerate}
Using fact 2) into~\eqref{eq:evolution-distance-average} one can rewrite it as
\begin{align}
	\x^{\ell+1} - \bar{\x}^{\ell+1} &= \Wm (\x^\ell - \bar{\x}^\ell) + \label{eq:evolution-distance-average-bis} \\ &- \alpha (\Im_{N} - \Em) \left( \nabla f(\x^\ell) + \tilde{\nabla} g(\x^{\ell+1}) \right). \nonumber
\end{align}
Moreover, by fact 1. it holds that $\x^\ell - \bar{\x}^\ell$ will always be perpendicular to the consensus space $\operatorname{span} \{ \1 \}$, and so one can write~\eqref{eq:evolution-distance-average-bis} as:
\begin{align}
	\x^{\ell+1} - \bar{\x}^{\ell+1} &= \bm{\Pi}_{\{ \1 \}^\perp} \Wm (\x^\ell - \bar{\x}^\ell) + \\ &- \alpha (\Im_{N} - \Em) \left( \nabla f(\x^\ell) + \tilde{\nabla} g(\x^{\ell+1}) \right). \nonumber
\end{align}

Taking the norm on both sides, and using the triangle inequality, one obtains
\begin{align*}
	&\norm{\x^{\ell+1} - \bar{\x}^{\ell+1}} \leq \norm{\bm{\Pi}_{\{ \1 \}^\perp} \Wm} \norm{\x^\ell - \bar{\x}^\ell} \\ &\qquad+ \norm{\Im_{N} - \Em} \norm{\alpha \left( \nabla f(\x^\ell) + \tilde{\nabla} g(\x^{\ell+1}) \right)} \\
	&\qquad \leq \rho(\Wm) \norm{\x^\ell - \bar{\x}^\ell} + \norm{\alpha \left( \nabla f(\x^\ell) + \tilde{\nabla} g(\x^{\ell+1}) \right)}
\end{align*}
where the fact that $\norm{\Im_{N} - \Em} = 1$ was used.
\end{proof}

\begin{lemma}[Bounded distance from solution]\label{lem:distance-from-solution}
Assume that the step size $\alpha$ satisfies~\eqref{eq:step-size-condition}. Then, the average has a bounded distance from the solution $\x^*$ to the original problem~\eqref{eq:time-invariant-problem}; in particular, 
\begin{equation}\label{eq:distance-optimal-solution}
	\norm{\bar{\x}^{\ell+1} - \x^*} \leq c \norm{\bar{\x}^\ell - \x^*} + \alpha L_f \norm{\x^\ell - \bar{\x}^\ell} + 2 \alpha L_g.
\end{equation}
\end{lemma}
\begin{proof}
For simplicity of exposition, consider the ``scalar'' average $\bar{x}^{\ell+1}$, characterized by the update
\begin{equation}\label{eq:scalar-average}
	\bar{x}^{\ell+1} = \bar{x}^\ell - \alpha \frac{\bm{1}_N^\top}{N} \left( \nabla f(\x^\ell) + \tilde{\nabla} g(\x^{\ell+1}) \right)
\end{equation}
where column stochasticity of $\Wm$ was used.

\noindent By the optimality condition of problem~\eqref{eq:time-invariant-problem}, it holds that $(\1_N^\top / N) \left( \nabla f(\x^*) + \tilde{\nabla} g(\x^*) \right) = 0$, and thus this term can be added to the right-hand side of~\eqref{eq:scalar-average}. Moreover, adding $x^*$ on both sides, taking the norm and using the triangle inequality yields:
\begin{equation}\label{eq:distance-average-to-optimal-solution}
\begin{split}
	&\norm{\bar{x}^{\ell+1} - x^*} \\ 
	&\leq \norm{\bar{x}^\ell - x^* - \alpha \frac{\bm{1}_N^\top}{N} \left( \nabla f(\bar{\x}^\ell) - \nabla f(\x^*) \right)} + \\
	&+ \alpha \norm{\frac{\bm{1}_N^\top}{N} \left( \nabla f(\x^\ell) - \nabla f(\bar{\x}^\ell) \right)} + \\
	&+ \alpha \norm{\frac{\bm{1}_N^\top}{N} \tilde{\nabla} g(\x^{\ell+1})} + \alpha \norm{\frac{\bm{1}_N^\top}{N} \tilde{\nabla} g(\x^*)}.
\end{split}
\end{equation}

The second through fourth terms on the right-hand side of~\eqref{eq:distance-average-to-optimal-solution} can be bound using the Lipschitz continuity of $\nabla f$ and $g$. Indeed:
\begin{equation*}
\begin{split}
	&\norm{\frac{\bm{1}_N^\top}{N} \left( \nabla f(\x^\ell) - \nabla f(\bar{\x}^\ell) \right)} \\ &\leq \norm{\frac{\bm{1}_N^\top}{N}} \norm{\nabla f(\x^\ell) - \nabla f(\bar{\x}^\ell)} \leq \frac{L_f}{\sqrt{N}} \norm{\x^\ell - \bar{\x}^\ell}
\end{split}
\end{equation*}
and, for any $\x$:
$$
	\norm{\frac{\bm{1}_N^\top}{N} \tilde{\nabla} g(\x)} \leq \norm{\frac{\bm{1}_N^\top}{N}} \norm{\tilde{\nabla} g(\x)} \leq \frac{L_g}{\sqrt{N}}.
$$

The square of the first term on the right-hand side~\eqref{eq:distance-average-to-optimal-solution} is now analyzed. By the definition of norm square it holds that
\begin{align*}
	&\norm{\bar{x}^\ell - x^* - \alpha \frac{\bm{1}_N^\top}{N} \left( \nabla f(\bar{\x}^\ell) - \nabla f(\x^*) \right)}^2 = \\
	&\quad = \norm{\bar{x}^\ell - x^*}^2 + \alpha^2 \norm{\frac{\bm{1}_N^\top}{N} \left( \nabla f(\bar{\x}^\ell) - \nabla f(\x^*) \right)}^2 + \\ &\quad - 2\alpha \langle \bar{x}^\ell - x^*, \frac{\bm{1}_N^\top}{N} \left( \nabla f(\bar{\x}^\ell) - \nabla f(\x^*) \right) \rangle
\end{align*}
and an upper bound the inner product is needed. By the fact that $\sum_i f_i / N$ is $m_f$-strongly convex and $L_f$-smooth, using \cite[Theorem~2.1.12]{nesterov_lectures_2018} one can derive
\begin{align*}
	\langle & \bar{x}^\ell - x^*, \frac{\bm{1}_N^\top}{N} \left( \nabla f(\bar{\x}^\ell) - \nabla f(\x^*) \right) \rangle \geq \\
	&\geq \frac{m_f L_f}{m_f + L_f} \norm{\bar{x}^\ell - x^*}^2 + \\ &+ \frac{1}{m_f + L_f} \norm{ \frac{\bm{1}_N^\top}{N} \left( \nabla f(\bar{\x}^\ell) - \nabla f(\x^*) \right)}^2 \, .
\end{align*}

Thus,
\begin{align*}
	&\norm{\bar{x}^\ell - x^* - \alpha \frac{\bm{1}_N^\top}{N} \left( \nabla f(\bar{\x}^\ell) - \nabla f(\x^*) \right)}^2 \leq \\
	&\leq \left( 1 - 2\alpha \frac{m_f L_f}{m_f + L_f} \right) \norm{\bar{x}^\ell - x^*}^2 + \\ &+ \alpha \left( \alpha - \frac{2}{m_f + L_f} \right) \norm{\frac{\bm{1}_N^\top}{N} (\nabla f(\bar{\x}^\ell) - \nabla f(\x^*))}^2.
\end{align*}
Notice that if $\alpha < 2 / (m_f + L_f)$ then the second term on the right-hand side is negative. Moreover, it holds that
$$
	0 < \left( 1 - 2\alpha \frac{m_f L_f}{m_f + L_f} \right) < 1 \quad \Leftrightarrow \quad 0 < \alpha < \frac{1}{2} \frac{m_f + L_f}{m_f L_f}.
$$
However, $2 / (m_f + L_f) < (1/2) (m_f + L_f) / (m_f L_f)$, and thus only ensuring~\eqref{eq:step-size-condition} is sufficient.

As a consequence, it follows that
\begin{align*}
	&\norm{\bar{x}^\ell - x^* - \alpha \frac{\bm{1}_N^\top}{N} \left( \nabla f(\bar{\x}^\ell) - \nabla f(\x^*) \right)}^2 \\ &\qquad \leq \left( 1 - 2\alpha \frac{m_f L_f}{m_f + L_f} \right) \norm{\bar{x}^\ell - x^*}^2
\end{align*}
and taking the square root and using the definition of $c$, one can write:
$$
	\norm{\bar{x}^\ell - x^* - \alpha \frac{\bm{1}_N^\top}{N} \left( \nabla f(\bar{\x}^\ell) - \nabla f(\x^*) \right)} \leq c \norm{\bar{x}^\ell - x^*}.
$$

Substituting these results back into~\eqref{eq:distance-average-to-optimal-solution} then yields:
$$
	\norm{\bar{x}^{\ell+1} - x^*} \leq c \norm{\bar{x}^\ell - x^*} + \alpha \frac{L_f}{\sqrt{N}} \norm{\x^\ell - \bar{\x}^\ell} + 2 \alpha \frac{L_g}{\sqrt{N}}
$$
and, by the fact that $\norm{\bar{\x}^{\ell+1} - \x^*} = \sqrt{N} \norm{\bar{x}^{\ell+1} - x^*}$, inequality~\eqref{eq:distance-optimal-solution} follows.
\end{proof}

\subsection{Convergence analysis}

Exploiting Lemmas~\ref{lem:implicit-update}--\ref{lem:distance-from-solution}, the proof of Proposition~\ref{pr:time-invariant-convergence} is presented next.
\begin{proof}[Proof of Proposition~\ref{pr:time-invariant-convergence}]
Using Lemma~\ref{lem:bounded-subgradient} into the result of Lemma~\ref{lem:distance-from-average} yields the inequality:
\begin{align}
	\norm{\x^{\ell+1} - \bar{\x}^{\ell+1}} &\leq \rho(\Wm) \norm{\x^\ell - \bar{\x}^\ell} + \label{eq:overall-distance-from-average} \\ &+ \alpha L_f \norm{\x^\ell - \tilde{\x}} + 2 \alpha L_g + \norm{(\Im - \Wm) \tilde{\x}}. \nonumber
\end{align}
Moreover, using the linear convergence of the proximal gradient method for strongly convex composite optimization, it follows that~\eqref{eq:linear-convergence-pgm} holds.

Using Lemma~\ref{lem:distance-from-solution} and the inequalities~\eqref{eq:overall-distance-from-average} and~\eqref{eq:linear-convergence-pgm}, one can then write $d^{\ell+1} \leq A d^\ell + b$, with $A$ and $b$ defined as in~\eqref{eq:linear-error-system}. Notice that $A$ is an upper triangular matrix with elements on the diagonal $c, \rho(\Wm), \zeta_\varphi \in (0,1)$; thus, $d^\ell$ is upper-bounded by the state of an asymptotically stable system with a constant input. Finally, using the triangle inequality, one has that:
\begin{align*}
	\norm{\x^{\ell+1} - \x^*} &\leq \norm{\bar{\x}^{\ell+1} - \x^*} + \norm{\x^{\ell+1} - \bar{\x}^{\ell+1}} \nonumber \\
	&= \begin{bmatrix} 1 & 1 & 0 \end{bmatrix} d^{\ell+1},
\end{align*}
which means that $\norm{\x^{\ell+1} - \x^*}$ is the output of a stable system with a fixed input.

Consider now the inexact update~\eqref{eq:time-invariant-DPGM}. The implicit update in Lemma~\ref{lem:implicit-update} therefore becomes, in the inexact case:
\begin{equation}\label{eq:inexact-implicit-update}
	\x^{\ell+1} = \Wm \x^\ell - \alpha \left( \nabla f(\x^\ell) + \tilde{\nabla} g(\x^{\ell+1} - \e^\ell) \right) + \e^\ell.
\end{equation}
Clearly,~\eqref{eq:implicit-update} and~\eqref{eq:inexact-implicit-update} differ for the error term $\e^\ell$, and for the fact that the subgradient is evaluated at $\x^{\ell+1} - \e^\ell$ in the latter. Evaluating the subgradient at $\x^{\ell+1} - \e^\ell$ does not affect the results of Lemmas~\ref{lem:bounded-subgradient}, \ref{lem:distance-from-average} and~\ref{lem:distance-from-solution}, since $\norm{\tilde{\nabla} g(\x^{\ell+1} - \e^\ell)} \leq L_g$.

Using the triangle inequality, one can see that the inexactness introduces the additional term $\norm{\e^\ell}$ in~\eqref{eq:linear-convergence-pgm},~\eqref{eq:distance-from-average}~and~\eqref{eq:distance-optimal-solution}, which yields~\eqref{eq:linear-error-system}. The result then follows by applying Assumption~\ref{as:stochastic-error} to the expected error evolution.
\end{proof}

\section{Proofs of section~\ref{sec:time-varying}}\label{app:time-varying}

\subsection{Proof of Proposition~\ref{pr:time-varying-convergence}}
Consider the inexact DPGM applied to the problem observed at time $t_{k+1}$. Under Assumption~\ref{as:time-varying}, by Proposition~\ref{pr:time-invariant-convergence} one has  that after $M_{\mathrm{o}}$ steps of~\eqref{eq:time-invariant-DPGM} the error is characterized by
\begin{align*}
	&\expval{d(t_{k+1})} \leq A^{M_{\mathrm{o}}} d^0  + \sum_{\ell = 0}^{M_{\mathrm{o}}-1} A^{M_{\mathrm{o}} - \ell - 1} \begin{bmatrix} 2 \alpha L_g \\ 2 \alpha L_g + \sigma' \\ 0 \end{bmatrix} \\ & \hspace{1.15cm} + \sum_{\ell = 0}^{M_{\mathrm{o}}-1} A^{M_{\mathrm{o}} - \ell - 1} \expval{\norm{\e^\ell}} \1_3 \\
	&\expval{\norm{\x(t_{k+1}) - \x^*(t_{k+1})}} \leq \begin{bmatrix} 1 & 1 & 0 \end{bmatrix} \expval{d(t_{k+1})},
\end{align*}
where the bound $\norm{(\Im - \Wm)\tilde{\x}(t_{k+1})} \leq \sigma'$ was used. By the warm-starting of Algorithm~\ref{alg:time-varying-DPGM}, one has that:
$$
	d^0 = \begin{bmatrix} \norm{\bar{\x}(t_k) - \x^*(t_{k+1})} \\ \norm{\x(t_k) - \bar{\x}(t_k)} \\ \norm{\x(t_k) - \tilde{\x}(t_{k+1})} \end{bmatrix},
$$
and using the triangle inequality and~\eqref{eq:sigma}, it is possible to get
\begin{align*}
	\expval{d^0} &\leq \begin{bmatrix} \norm{\bar{\x}(t_k) - \x^*(t_k)} \\ \norm{\x(t_k) - \bar{\x}(t_k)} \\ \norm{\x(t_k) - \tilde{\x}(t_k)} \end{bmatrix} + \begin{bmatrix} \norm{\x^*(t_{k+1}) - \x^*(t_k)} \\ 0 \\ \norm{\tilde{\x}(t_{k+1}) - \tilde{\x}(t_k)} \end{bmatrix} \\
	&\leq \expval{d(t_k)} + \sigma [1, 0, 1]^\top. \hspace{4.75cm}\square
\end{align*}

\subsection{Proof of Corollary~\ref{cor:asymptotical-error}}
By definition, the diagonal elements of $A$ are upper bounded by $\delta$, thus the error~\eqref{eq:time-varying-error-system} can be upper bounded as 
\begin{equation}\label{eq:upper-bound-asymptotical-error}
	\expval{d(t_{k+1})} \leq \delta^{M_{\mathrm{o}}} \expval{d(t_k)} + \sum_{\ell = 0}^{k} \delta^{M_{\mathrm{o}} (k - \ell)} b''.
\end{equation}
Then, iterating~\eqref{eq:upper-bound-asymptotical-error} and taking the limit for $k \to \infty$ yields
$$
	\limsup_{k \to \infty} \expval{d(t_k)} \leq \frac{1}{1 - \delta^{M_{\mathrm{o}}}} b''
$$
which implies the desired result using the fact that $\expval{\norm{\x(t_{k+1}) - \x^*(t_{k+1})}} \leq \begin{bmatrix} 1 & 1 & 0 \end{bmatrix} \expval{d(t_{k+1})}$ and the definition of $b''$. \qed

\fi

\bibliographystyle{IEEEtran}
\bibliography{IEEEabrv,references}

\end{document}